\documentclass[onefignum,onetabnum,final]{siamart220329}

\usepackage{amsmath,amssymb,enumitem}
\usepackage{stmaryrd}
\usepackage{txfonts}

\newtheorem{rmk}{Remark}
\newtheorem{defn}{Definition}

\newcommand{\vx}{\bar{x}}
\newcommand{\vX}{\Bar{X}}
\newcommand{\vY}{\Bar{Y}}
\newcommand{\vxN}{\vx^N}
\newcommand{\vy}{\bar{y}}

\newcommand{\unf}{u^{N,f}}

\newcommand{\R}{\mathbb{R}}
\newcommand{\N}{\mathbb{N}}
\newcommand{\T}{\mathbb{T}}

\newcommand{\pt}{\partial_t}

\newcommand{\intd}{\int_{\R^d}}

\newcommand{\Ue}{U^{\epsilon}}

\newcommand{\tbphi}{\Tilde{b}^{\phi }}
\newcommand{\vZ}{\Bar{Z}}
\newcommand{\un}{u^{N,\epsilon}}
\newcommand{\unt}{\Tilde{u}^{N,\epsilon}}

\newcommand{\unphit}{\Tilde{u}^{N,\phi,\epsilon}}

\newcommand{\bn}{b^N}

\newcommand{\dive}{\text{div}}

\newcommand{\W}{\mathcal{W}}
\newcommand{\Prob}{\mathbb{P}}

\title{A Mean Field Game Approach to Large Deviations for Empirical Measures under common noise.}
\author{Nikiforos Mimikos-Stamatopoulos\thanks{University of Chicago, Department of Mathematics, Chicago, IL}}

\makeatletter
\def\namedlabel#1#2{\begingroup
    #2%
    \def\@currentlabel{#2}%
    \phantomsection\label{#1}\endgroup
}
\makeatother
\begin{document}
\maketitle
\begin{abstract}
We present a Mean Field Game approach to obtain the rate function for the empirical measure of interacting particles under McKean-Vlasov dynamics. Although the result is well known, our approach relies on PDE methods and provides another example of the well known connection between Mean Field Games and Large Deviations.  
\end{abstract}
\section{Introduction}
The main purpose of this paper is to show a Large Deviations Principle (LDP for short), for the empirical measure of particles interacting through McKean-Vlasov dynamics, using a PDE approach. Specifically, we consider a system of $N-$particles governed by the following dynamics
\begin{equation}\label{eq:DynamicsParticles}
    \begin{cases}
     dX_t^{i,N}= b(X_t^{i,N},m_{\vX_t^N}^N)dt+dW_t^i +\sqrt{\alpha_0}dW^0,\\
     X_0^{i,N}=x_i^N \text{ for each }i\in \{1,\cdots,N\},
    \end{cases}
\end{equation}
where the initial positions $x_i\in \R^d$ are given, $W^0,W_{\cdot}^1,\cdots,W_{\cdot}^N$ are independent standard Brownian motions, $b:\R^d\times \mathcal{P}(\R^d)\rightarrow \R$ is a given vector field, $\alpha_0\geq 0$ is the intensity of the common noise, and
\begin{equation}
    \label{defn: EmpiricalDensityIntro}
    m_{\vX_t^N}^N=\frac{1}{N}\sum\limits_{i=1}^N \delta_{X_t^{i,N}}
\end{equation}
is the empirical measure.\par

In the absence of common noise, that is when $\alpha_0 = 0$, Gartner and Dawson in \cite{dawsont1987large}, established a LDP for $m_{\vX_t^N}^N$ using probabilistic techniques. In the presence of common noise, Delarue, Lacker and Ramanan in \cite{Lacker1}, extended these results, again through probabilistic arguments. In this paper, we present a direct analytical proof of the aforementioned results. We do this using tools and ideas from Mean Field Games (MFG for short), which avoids the use of heavy probabilistic arguments.\par

\subsection{Related and known results}

The connection between MFG and LDP's has been well established, initially by J.-M. Lasry and P.-L. Lions in \cite{lasry2007mean}. Further developments may be found for example in \cite{cecchin2019convergence,saldi2021large,daudin2023mean,Lacker1,Lacker2}.  Moreover, it should be noted that there exists a large literature concerning PDE approaches to LDP's, for example \cite{FlemmingSouganidis}, \cite{EvansSouganidis}, \cite{LDPHilbert}, \cite{FlemmingLDPPDE}, \cite{FengHJ} and \cite{evans1985pde}, but these works do not cover the case of LDP's related to empirical measures.\par 
Existing PDE approaches to this problem, as far as we can tell, only address the case $\alpha_0=0$. Notably, the works of Katsoulakis and Feng in \cite{feng2009comparison} and Feng and Kurtz in \cite{feng2006large}, do establish a PDE approach to LDP's for general settings, including the one considered in this paper. Their approach relies on the analysis of Hamilton Jacobi equations (HJ for short), in infinite dimensional spaces. Our goal however is to show that in the particular case of a LDP for empirical measures driven by McKean-Vlasov dynamics, the results may be obtained more directly using MFG. Moreover, a MFG approach to a similar problem has been implemented by S. Daudin in \cite{daudin2023mean}. Namely, the work in \cite{daudin2023mean} establishes an LDP for the exit time problem of the empirical density, by applying results of MFG with constraints. For the case $\alpha_0>0$ as mentioned above the PDE approach we present appears to be new.

\subsection{Main Results}
We now briefly and somewhat informally describe the results we wish to replicate. In subsection \ref{subsection: Case a = 0} we recall the results from \cite{dawsont1987large}, while in subsection \ref{subsection: Case a > 0} the results from \cite{Lacker1}. To keep the presentation compact, we defer some of the definitions to Section \ref{Section: Notation}. In what follows we always deal with $N-$particles governed by \ref{eq:DynamicsParticles} and let $\vX^N = (X^{1,N},\cdots, X^{N,N})$. Furthermore, we assume that the initial distribution of the particles is such that
\[T_N(x_1,\cdots,x_N):=\frac{1}{N}\sum\limits_{i=1}^N\delta_{x_i}\rightarrow \nu \text{ in }\mathcal{P}_2(\R^d),\]
for some measure $\nu\in \mathcal{P}_2(\R^d)$.\par

\subsubsection{Case $\alpha_0=0$}\label{subsection: Case a = 0}
In the absence of common noise, it was established in \cite{dawsont1987large}, that the process \[m_{\vX_{\cdot}^N}^N:=\frac{1}{N}\sum\limits_{i=1}^N\delta_{X_{\cdot}^{i,N}},\]
satisfies a LDP on the space $C([0,T];\mathcal{P}_1(\R^d))$, with rate function
\begin{equation}\label{eq: GartnerDawsonRateFuntion}
    I_{GD}(\mu(\cdot)):=
    \begin{cases}
    \int_0^T \|\pt \mu-\mathcal{L}_{\mu(t)}^*(\mu(t))\|_{\mu(t)}^2 dt \text{ if }\mu(\cdot ) \text{ is absolutely continuous},\\
    +\infty \text{ otherwise.}
    \end{cases}
\end{equation}
Here, for a measure $\mu\in \mathcal{P}(\R^d)$,  $\mathcal{L}_{\mu}^*$ is the formal adjoint of the operator 
\[\mathcal{L}_{\mu}:=\frac{1}{2}\Delta +b(\cdot,\mu)\cdot D, \]
and for a distribution $\theta \in D'(\R^d)$
\[\|\theta\|_{\mu}^2:=\sup\limits_{f\in C^1(\R^d)}\{\langle \theta, f\rangle -\frac{1}{2}\langle \mu,|Df|^2\rangle  \}.\]
We now present our main result in the absence of common noise which involves two cases. In the first case we are describing the behaviour of the empirical density as a process on $C([0,T];\mathcal{P}_1(\R^d))$, while in the second case we are looking at the behaviour of the empirical density at the terminal time $T$. For the convenience of the reader we provide the relevant definitions, such as weak and strong LDP in Subsection \ref{subsection: Prerequisites}.
\begin{theorem}\label{thm:MainTheorem}
Assume \ref{assum: AssumptionOnb}. Fix $t_0\in [0,T)$, $\vx^N=(x_1^N,\cdots,x_N^N)\in (\R^d)^N$, and $\nu \in \mathcal{P}_2(\R^d)$ such that 
\[\lim\limits_{N\rightarrow \infty}\mathcal{W}_2(T_N(\vx^N),\nu)=0.\]
Let $X_{\cdot}^{i,N}, i\in \{1,\cdots, N\}$ be given by \eqref{eq:DynamicsParticles} with $X_{t_0}^{i,N}=x_{i}^N$ and let $m_{\vX_t^N}^N$ be the empirical distribution given by 
\[m_{\vX_t^N}^N:=\frac{1}{N}\sum\limits_{i=1}^N \delta_{X_t^{i,N}}.\]
Then, the process $m_{\vX_{\cdot}^N}^N$  satisfies a weak-LDP in $C([t_0,T];\mathcal{P}_1(\R^d))$ with rate function 
\begin{equation}\label{eq: FormulaForRateFunction}
I_{\text{path}}(\mu(\cdot))=
\begin{cases}
    \inf\limits_{{\alpha}\in \mathcal{B}_{\mu(\cdot )}}\Big[ \frac{1}{2}\int_{t_0}^T \langle \mu_s^{\alpha},|{\alpha}_s|^2\rangle ds \Big] \text{ if }\mathcal{B}_{\mu(\cdot )}\neq \emptyset,\\
    +\infty \text{ otherwise,}
\end{cases}
\end{equation}
where $\mathcal{B}_{\mu(\cdot)}$ is the (possibly empty) set of all controls ${\alpha}$ such that
\begin{equation*}
    \begin{cases}
    \pt \mu(t)+\mathcal{L}_{\mu(t)}^*\mu(t)-\text{div}({\alpha}\mu)=0,\\
    \mu(0)=\nu,
    \end{cases}
\end{equation*}
in the distributional sense. Moreover, the empirical measures at the terminal time
\[m_{\vX_T^N}^N:=\frac{1}{N}\sum\limits_{i=1}^N \delta_{X_T^{i,N}},\] 
satisfy an LDP on $\mathcal{P}_1(\R^d)$ with $\text{good}$ rate function 
\begin{equation}\label{eq: RatefunctionInTheoremTerminalTime}
I(t_0,\mu_T)=
\begin{cases}
    \inf\limits_{\alpha\in \mathcal{A}_{\nu,\mu_T}}\Big[ \frac{1}{2}\int_{t_0}^T \langle \mu_s^{\alpha},|{\alpha}_s|^2\rangle ds \Big] \text{ if }\mathcal{A}_{\nu,\mu_T}\neq \emptyset,\\
    +\infty \text{ otherwise,}
\end{cases}
\end{equation}
where $\mathcal{A}_{\nu,\mu_T}$ is the (possibly empty) set of all controls ${\alpha}$ such that 
\begin{equation}\label{eq: ControlsDefTerminal}
    \begin{cases}
    \pt \mu_s^{\alpha}+\mathcal{L}_{\mu_s^{\alpha}}^*(\mu_s^{\alpha})-\text{div}({\alpha}\mu^{\alpha}) =0,\\
    \mu(0)=\nu, \mu(T)=\mu_T,
    \end{cases}
\end{equation}
in the distributional sense.
\end{theorem}
\begin{rmk}
    Although in the statement of Theorem \ref{thm:MainTheorem} we presented first the LDP for the path of empirical measures and the LDP of the terminal time as a consequence, this was only done for presentational and conceptual purposes. In reality we first show the LDP for the terminal time and the LDP for the path is a consequence.
\end{rmk}
\begin{rmk}
To connect expressions \eqref{eq: GartnerDawsonRateFuntion} and \eqref{eq: FormulaForRateFunction}, we provide the following informal argument. We start with expression \eqref{eq: GartnerDawsonRateFuntion} and assume that we are given a path $\mu(\cdot)\in C([0,T];\mathcal{P}_2(\R^d))$, such that $I_{\text{path}}(\mu(\cdot))<\infty$. From the definition of $I_{\text{path}}$, we infer the existence of a control ${\alpha}$ such that 
\[\pt \mu -\mathcal{L}_{\mu(t)}^*(\mu(t))-\dive ({\alpha}\mu)=0 .\]
Therefore, 
\[I_{GD}(\mu(\cdot))=\int_0^T \|\pt \mu -\mathcal{L}_{\mu(t)}^*(\mu(t))\|_{\mu(t)}^2= \int_0^T \sup\limits_{f \in C^1(\R^d)} \{\langle \pt \mu -\mathcal{L}_{\mu(t)}^*(\mu(t)),f \rangle -\frac{1}{2}\langle \mu ,|Df|^2\rangle dt\}\]
\[= \int_0^T \sup\limits_{f\in C^1(\R^d)}\{\langle -Df \cdot {\alpha} -\frac{1}{2}|Df|^2,\mu\rangle dt\}=\int_0^T \frac{1}{2}\langle |{\alpha}_t|^2,\mu(t)\rangle dt = I_{\text{path}}(\mu(\cdot )).\]
 The above heuristic is justified in \cite{dawsont1987large}.
\end{rmk}
\begin{rmk}
It is important to point out the different use of the spaces $\mathcal{P}_1(\R^d)$ and $\mathcal{P}_2(\R^d)$ throughout the paper. Theorem \ref{thm:MainTheorem} shows a LDP on the space $\mathcal{P}_1(\R^d)$, however, we point to the fact that the rate functions $I,I_{\text{path}}$, will only be finite on the subset $\mathcal{P}_2(\R^d)$. This condition guarantees that $I$ is a good rate function, as defined in Section \ref{Section: Notation}. Furthermore, this is to be expected from our assumptions on the initial density. Indeed, taking for simplicity $t_0=0$, an initial density $\nu \in \mathcal{P}_2(\R^d)$ and a terminal density $\mu_T\in \mathcal{P}_1(\R^d)$ such that $I(0,\mu_T)<\infty$, we can deduce from \eqref{eq: RatefunctionInTheoremTerminalTime} the existence of a control $\alpha $ such that \eqref{eq: ControlsDefTerminal} holds and 
\[\frac{1}{2}\int_0^T \langle \mu_s^{\alpha},|\alpha_s|^2\rangle <\infty. \]
Thus, testing against $|x|^2$ in \eqref{eq: ControlsDefTerminal} yields
\[\pt \intd |x|^2 d\mu_s^{\alpha} =2d+\intd 2 x\cdot \alpha d\mu_s^{\alpha}\leq d +\intd |x|^2 d\mu_s^{\alpha} +\intd |a_s|^2 d\mu_s^{\alpha},  \]
which shows that 
\[\sup\limits_{t\in [0,T]}\intd |x|^2d \mu_t^{\alpha}<\infty,\]
by Gr\"onwall. Therefore, we have that $\{I(0,\cdot)<\infty\}\subset \mathcal{P}_2$.
\end{rmk}
\subsubsection{Case $\alpha_0 > 0$}\label{subsection: Case a > 0}
To present the results in \cite{Lacker1}, we need to introduce some notation. Given $x\in \R^d$, let $\tau_x :\R^d\rightarrow \R^d$ be the translation map, that is $\tau_x(y)= y-x$. Given a path $\phi \in C([0,T];\R^d)$ let $\Tilde{b}(t,x,m) = b(x +\phi_t, m\circ \tau_{\phi_t})$ and let $\Tilde{I}_{GD}$ be the rate function $I_{GD}$ with the operator $\Tilde{\mathcal{L}} = -\frac{1}{2}\Delta +\Tilde{b}\cdot D$, that is the drift $b$ is replaced by $\Tilde{b}$. Moreover, for $\mu \in C([0,T];\mathcal{P}(\R^d))$, let $I^{\phi}(\mu) = \Tilde{I}_{GD}(\mu \circ \tau_{\phi_{\cdot}})$. Then, it was established in \cite{Lacker1}, that the empirical density $m_{\vX_{\cdot}^N}^N$ satisfies a weak-type LDP in $C([0,T];\mathcal{P}_1(\R^d))$, with rate function 
\begin{equation}\label{defn: Rate FUnction of Lacker}
    J^{\alpha_0} (\mu) = \inf\limits_{\phi \in C^0([0,T];\R^d)}I^{\sqrt{\alpha_0}\phi}(\mu )
\end{equation}
and in fact the infimum is achieved at a $\phi_0$ satisfying
\[\phi_0(t) = \intd x d(\mu_t-\mu_0)(x) -\int_0^t \intd \Tilde{b}(s,x, \mu_s) d\mu_s(x)ds. \]
We refer to \cite{Lacker1} for the exact result.
In this case our result is the following.
\begin{theorem}
    Let b satisfy \ref{assum: AssumptionOnb}. Let $\vx^N = (x_1^N,\cdots, x_N^N)\in (\R^d)^N$ and $\vX^N = (X^{1,N},\cdots ,X^{N,N})$ be given by 
    \begin{equation}
\begin{cases}
        d X_s^i = b(X_s^i,m_s^N)ds+dW_s^i+\sqrt{\alpha_0} dW_s^0,\\
        X_t^i= x^i.
\end{cases}
\end{equation}
Finally, assume that for some measure $\nu \in \mathcal{P}_2(\R^d)$, we have that
\[\lim\limits_{N\rightarrow \infty} \mathcal{W}_2(m_{\vx^N}, \nu) = 0.\]
Then, the path of the empirical measure 
\[m_{\vX_{\cdot}^N}^N \in C([0,T];\mathcal{P}_1(\R^d)),\]
satisfies a weak LDP in $C([0,T];\mathcal{P}_1(\R^d))$, with rate function 
\begin{equation}\label{defn: Upath inside theorem}
    U_{\text{path}}(\mu_{\cdot }) =\inf\limits_{\phi \in C^0}U_{\text{path}}^{\sqrt{\alpha_0}\phi}(\mu_{\cdot}).
\end{equation}
Where $U^{\phi}(\mu(\cdot)) = I_{\text{path}}(\mu\circ \tau_{\phi})$, with $I_{\text{path}}$ being the rate function from the case $\alpha_0 = 0$. Moreover, the empirical distribution at the terminal time $\mu_{\vX_T^N}^N$, satisfies a weak LDP with rate function 
\[U(t_0,\mu_T)= \inf\limits_{\phi \in C^0}U^{\sqrt{\alpha_0}\phi}(t_0,\mu_T),\]
where $U^{\phi}(t_0,\mu_T) = I(t_0,\mu_T\circ  \tau_{\phi})$.
\end{theorem}
As explained in \cite{Lacker1} when common noise is present we no longer have a good rate function and thus only a weak LDP may be obtained. The connection between \eqref{defn: Rate FUnction of Lacker} and \eqref{defn: Upath inside theorem} follows in the same way as the case of $\alpha_0 = 0$.
\begin{rmk}
    In both \cite{dawsont1987large} and \cite{Lacker1}, roughly speaking, the strategy is to first obtain an LDP through a Sanov type Theorem. The rate function obtained from this step however is quite involved in its description. The difficult part is to then find a more intuitive formula for this rate function. In our approach almost the opposite is true. Due to the PDE approach a control like description is readily available to us for the candidate rate function. The main difficulty in our case is establishing that an LDP holds. Finally, we note that although our results verify the ones in \cite{dawsont1987large} and \cite{Lacker1}, both these works establish LDP's with more general assumptions on the data.
\par 
\end{rmk}
\section{Notation, Assumptions and Prerequisites}\label{Section: Notation}
\subsection{Notation}
In the following $d,N\in \mathbb{N}$ and $T>0$. We denote by $\mathcal{P}(\R^d)$ the space of Borel probability measures on $\R^d$. Furthermore, for $p\geq 1$, $\mathcal{P}_p(\R^d)$ is the set of $m\in \mathcal{P}(\R^d)$ such that
\[\int_{\R^d}|x|^p\nu(dx)<\infty.\]
We equip $\mathcal{P}_p(\R^d)$ with the standard Wasserstein $p-$distance which we denote by $\mathcal{W}_p$. For $k\in \N$ and $\Omega\subset \R^d$, we will use the notation $C^k(\Omega)$ for the space of functions with continuous derivatives up to order $k$ and $C_c^k(\Omega)$ for the functions in $C^k(\Omega)$ with compact support. Moreover, for $k\in \N$ we denote by $C^k([0,T];\R^d)$ the space of continuous paths in $\R^d$ with continuous derivatives up to order $k$, equipped with the norm $\|\phi\|_k = \sum\limits_{i=0}^k\|D^i \phi\|_{\infty}$. Furthermore, we denote by $H^1([0,T]; \R^d)$ the usual Sobolev space of functions with bounded norm $\|\phi\|_{H^1}^2 = \|\Dot{\phi}\|_2^2$ and $\phi(0)= 0$. We also denote by $C^0([0,T];\R^d)$ the space of continuous paths $\phi: [0,T]\rightarrow \R^d$ such that $\phi(0)=0$.\\
We will frequently use the notation $\vx^N=(x_1,\cdots,x_N)$ to denote points in $(\R^d)^N$. Given a vector $\vx^N=(x_1,\cdots,x_N)\in (\R^d)^N$ we define the operators $T_N:(\R^d)^N\rightarrow \mathcal{P}(\R^d)$ by
\[T_N(\vx^N)=T_N(x_1,\cdots,x_N)=\frac{1}{N}\sum\limits_{i=1}^N \delta_{x_i}\in \mathcal{P}(\R^d).\]
Given a measure $\mu \in \mathcal{P}(\R^d)$ and a map $\psi :\R^d\rightarrow \R^d$, we denote by $m\circ \psi$ the measure, given by $\mu\circ \psi (A) = \mu (\psi^{-1}(A))$ for all Borel sets $A$. For a function $U:[0,T]\times \mathcal{P}_2(\R^d)\rightarrow \R$ we denote by $\delta_m U,D_m U:[0,T]\times \mathcal{P}_2(\R^d)\times \R^d\rightarrow \R$ the derivatives as defined in \cite{cardaliaguet2019master}. Finally, we denote by $C^{0,1}(\mathcal{P}_1(\R^d))$ the functions that are Lipschitz continuous in $\W_1$.\\
\textbf{Convention:} Throughout the paper we will use the notation $T_N(\vx^N)$ for points in $(\R^d)^N$, while for processes $\vX_t^N$ we will denote their empirical density by $m_{\vX_t^N}^N$, that is 
\[m_{\vX_t^N}^N:= \frac{1}{N}\sum\limits_{i=1}^N \delta_{\vX_t^{i,N}}.\]
\subsection{Assumptions}
For the drift $b:\R^d\times \mathcal{P}_2(\R^d)\rightarrow \R^d$ we assume that there exists a $R>0$ such that for any $m\in \mathcal{P}_2(\R^d)$
\begin{equation}\tag{[Hb]}\label{assum: AssumptionOnb}
    \begin{cases}
b(\cdot, m) \in C_c^2(B(0,R)),\\
\frac{\delta b}{\delta m}b(x,m,y)=(\frac{\delta b}{\delta m}(x,m, y),\cdots,\frac{\delta b^n}{\delta m}(x,m, y)) \in C_c^2(B(0,R)).
\end{cases}
\end{equation}
\begin{rmk}
    It is worth mentioning that our assumptions on $b$ have not been optimized. 
\end{rmk}
\subsection{Prerequisites}\label{subsection: Prerequisites}
In this subsection, for the readers convenience, we collect the definitions and the main results we will require from the theory of Large Deviations. For a more detailed study however, they can also be found for example in \cite{feng2006large}.\\
In the rest of the subsection $\{X_n\}_{n\in \N}$ is a sequence of random variables with values in a metric space $(S,d)$. 
\begin{defn}
The sequence $\{X_n\}_{n\in \N}$ satisfies a large deviation principle if there exists a lower semicontinuous function $I:S\rightarrow [0,\infty],$ such that for each open set $A$,
\begin{equation}\label{defn: LowerBoundLDP}
    -\inf\limits_{x\in A} I(x)\leq \liminf\limits_{n\rightarrow \infty} \frac{1}{N}\log\Big(\mathbb{P}[X_n\in A]\Big),
\end{equation}
and for each closed set $B$
\begin{equation}\label{defn: UpperBoundLDP}
 \limsup\limits_{n\rightarrow \infty} \frac{1}{n}\log \Big(\mathbb{P}[X_n\in B]\Big)\leq -\inf\limits_{x\in B} I(x).
\end{equation}
We say that $I$ is the rate function for the LDP. The rate function is good if for each $c\in [0,\infty)$ the set $\{I\leq c\}$ is compact. Finally, we say that $\{X_n\}_{n\in \N}$ satisfies a weak LDP if condition \eqref{defn: UpperBoundLDP} only holds for compact sets.
\end{defn}
\begin{proposition}\label{prop: FromWeakToStrong}
    Assume that $\{X_n\}_{n\in \N}$ satisfies a weak LDP with a good rate function $I:S\rightarrow [0,\infty]$. Then $\{X_n\}_{n\in \N}$ satisfies a LDP with rate function $I$.
\end{proposition}

\begin{proposition}\label{prop: EpsilonZeroJustification}
    Assume that for a rate function $I:S\rightarrow \R$, the sequence $\{X_n\}_{n\in \N}$ satisfies for all $x_0\in S$
    \begin{equation}
        \lim\limits_{\epsilon\rightarrow 0^+}\lim\limits_{n\rightarrow \infty}\frac{1}{n}\log \Big(\mathbb{P}[X_n \in B(x_0,\epsilon)]\Big) =- I(x_0).
    \end{equation}
    Then, $\{X_n\}_{n\in \N}$ satisfies a weak LDP with rate function $I$. Moreover, if $I$ is a good rate function then $\{X_n\}_{n\in \N}$ satisfies a LDP.
\end{proposition}
The following elementary Lemma will be useful in latter computations.
\begin{lemma}\label{lem: ReplaceTheLimit}
    Let $\{a_n\}_{n\in \N},\{b_n\}_{n\in \N}$ be two positive sequences such that 
    \[\lim\limits_{n\rightarrow \infty} a_n = \lim\limits_{n\rightarrow \infty} b_n.\]
    Then, 
    \[\lim\limits_{n\rightarrow \infty}\frac{1}{n}\log(a_n)=\lim\limits_{n\rightarrow \infty}\frac{1}{n}\log(b_n).\]
\end{lemma}

\section{Main Steps of the proof}
In this section we provide a sketch of the main arguments for the proof of Theorem \ref{thm:MainTheorem}. Recall, that the processes $\vX^N$ satisfies
\begin{equation}\label{eq: EquationForTheX^i's}
\begin{cases}

dX_s^{i,N} = b(X_s^{i,N},m_{\vX_s^N}^N)ds+dW_s^i,\\
X_t^{i,N}=x_i^N,
\end{cases}
\end{equation}
where $m_{\vX_s^N}^N=\frac{1}{N}\sum\limits_{i=1}^N \delta_{X_s^{i,N}}$ is the empirical measure. For a fixed $\epsilon>0$ and a given measure $\mu_T\in \mathcal{P}(\R^d)$, following the standard PDE approach, we will look directly at the functions 
\begin{equation}\label{eq: DefinitionOfUNEWITHMEASURE}
    w^{N,\epsilon}(t,x_1,\cdots,x_N) =-\frac{1}{N}\log(\mathbb{P}[\mathcal{W}_1(T_{N}(X_T^{1,N},\cdots,X_T^{N,N}),\mu_T)\leq \epsilon|(X_t^{1,N},\cdots,X_t^{N,N})=(x_1,\cdots,x_N)]),
\end{equation}
and identify the limit
\begin{equation}\label{eq: limitequationforMainSteps}
    \lim\limits_{\epsilon\rightarrow0^+}\lim\limits_{N\rightarrow \infty}w^{N,\epsilon}(t,x_1,\cdots,x_N),
\end{equation}
which would characterize the rate function by Proposition \ref{prop: EpsilonZeroJustification}. However, in order to simplify some bounds, instead of fixing a measure $\mu_T\in \mathcal{P}(\R^d)$, it is better to consider a fixed sequence $z^N=(z_1^N,\cdots,z_N^N)\in (\R^d)^N$ such that 
\begin{equation}\label{eq: DefnOfz^N}
    \lim\limits_{N\rightarrow \infty}\mathcal{W}_2(T_N(z^N),\mu_T)=0,
\end{equation}
and instead of \eqref{eq: DefinitionOfUNEWITHMEASURE}, consider the functions
\begin{equation}\label{fun: Definition Of un}
    \un(t,x_1,\cdots,x_N)=-\frac{1}{N}\log(\mathbb{P}[\mathcal{W}_1(m_T^N,T_N(z^N))\leq \epsilon|(X_t^1,\cdots,X_t^N)=(x_1,\cdots,x_N)]).
\end{equation}
By Lemma \eqref{lem: ReplaceTheLimit}, this would provide us with the same result .\\
We now describe the key estimates for the functions $\un$. We start with the bounds for letting $N\rightarrow \infty$ while keeping $\epsilon>0$ fixed. In what follows $C>0$, will denote a constant that may depend on $b,T,\epsilon$ and the dimension $d$, but importantly not on $N$.
\begin{enumerate}
    \item First we show that the functions $\un$ as defined in \eqref{fun: Definition Of un}, satisfy a bound of the form
    \[0\leq \un(t,x_1,\cdots,x_N) \leq \frac{C}{T-t}(\mathcal{W}_2(T_N(x_1,\cdots,x_N),T_N(z^N))^2+1).\]
    \item By applying a Bernstein type argument we then show 
    \[|D \un(t,\vxN)|^2\leq \frac{C}{N(T-t)}\Big(\frac{\un(t,\vxN)}{N}+1\Big), \]
    which after inserting the bound of $\un$ yields
    \[|D\un(t,\vxN)|\leq \sqrt{\frac{C}{N}}\Big(\frac{\mathcal{W}_2(T_N(\vxN),T_N(z^N))+1}{T-t}\Big).\]
    \item With the above bound on the gradient, we have that for all $t\in [0,T)$ and bounded $K\subset\mathcal{P}_2(\R^d)$, there exists a constant $C=C(K,T-t,b,d,\epsilon)>0$ such that 
    \[|\un(t,\vx^N)-\un(t,\vy^N)|\leq C\mathcal{W}_2(T_N(\vx^N),T_N(\vy^N)),\]
    for all $\vx^N=(x_1,\cdots,x_N),\vy^N =(y_1,\cdots,y_N)\in \R^d$ with $T_N(\vx^N),T_N(\vy^N)\in K$.
    \item Finally, through a semiconcavity estimate we obtain a lower bound on $\pt \un$ and we may then pass to the limit.
\end{enumerate}
Once we have passed to the limit in $N$, we rely on results regarding the control formulation of the solutions to the Master equation in MFG, in order to identify the limit as we let $\epsilon\rightarrow 0^+$.
\section{Estimates on the $N-$particle HJB equation}
In this section we establish bounds for the functions $\un$ as defined in \eqref{fun: Definition Of un}. The functions $\un$ formally satisfy 
\begin{equation}\label{eq: DefinitionOfu^N,epsilon}
    \begin{cases}
    -\pt \un -\frac{1}{2}\Delta \un+\frac{N}{2}|D\un|^2+\bn \cdot D \un =0 \text{ in }[0,T]\times (\R^d)^N,\\
    \un(T,x):=\begin{cases}
    0 \text{ if }\mathcal{W}_1(T_N(\vx^N),T_N(z^N))\leq \epsilon,\\
    +\infty \text{ otherwise.}
    \end{cases}
    \end{cases}
\end{equation}
where 
\[b^N(t,\vx^N)=(b(x_1,T_N(\vxN)),\cdots, b(x_N,T_N(\vxN))).\] Due to the degenerate terminal condition, our approach is to first obtain bounds for the functions
\begin{equation}\label{eq: DefinitionOfuNF}
    u^{N,f}(t,\vx^N)=-\frac{1}{N}\log(\mathbb{E}[e^{-Nf(T_N(X_T^1,\cdots,X_T^N))}|(X_t^1,\cdots,X_t^N)=\vx^N]),
\end{equation}
where $f:\mathcal{P}_1(\R^d) \rightarrow \R$ is Lipschitz. The functions $\unf$ satisfy
\begin{equation}\label{eq: EquationUnf}
\begin{cases}
-\pt \unf -\frac{1}{2}\Delta \unf+\frac{N}{2}|D\unf|^2+\bn \cdot D \unf =0 \text{ in }[0,T]\times (\R^d)^N,\\
    \unf(T,\vx^N)=f(T_N(\vx^N)) \text{ in }(\R^d)^N.
\end{cases}
\end{equation}
Then, by considering a sequence of Lipschitz functions $(f_{\delta})_{\delta>0}:\mathcal{P}(\R^d)\rightarrow \R$ that pointwise approximate the function
\begin{equation}
    f_0(\nu):=\begin{cases}
    0 \text{ if }\mathcal{W}_1(T_N(x),T_N(z))\leq \epsilon,\\
    +\infty \text{ otherwise,}
    \end{cases}
\end{equation}
and using the fact that for each fixed $N\in \N$,
\[\lim\limits_{\delta \rightarrow 0}u^{N,f_{\delta}}(t,\vx^N)= \un(t,\vx^N),\]
as is readily seen by \eqref{eq: DefinitionOfuNF}, we will obtain bounds for $\un$.\\
In the following, for $k\in \N$, a vector $z\in (\R^{d})^k$ and $\delta>0$ we define the set
\begin{equation*}
    \mathcal{A}_k(z,\delta):=\{f\in C^{0,1}(\mathcal{P}(\R^d);\R): f(\nu)=0, \text{ for all }\nu\in \mathcal{P}(\R^d): \mathcal{W}_1(\nu,T_k(z))\leq \delta, f\geq 0\}.
\end{equation*}
Furthermore, in the proofs of this section $C>0$ is a constant that is independent of $N$ and may change from line to line.
\begin{proposition}\label{Prop:UpperBoundOnUnIndividualParticle}
Let $z\in \R^d, \epsilon>0$, $f\in \mathcal{A}_1(z,\epsilon)$. Define $u:[0,T]\times \R^d\rightarrow \R$ by
\[u(t,x):=-\log(\mathbb{E}[e^{-f(X_T)}|X_t=x]),\]
where $dX_s =dW_s$, for standard Brownian motion $W_{\cdot}$. Then, there exists a constant $C=C(b,T,d)>0$ and a function $c:\N\times\R\times \R \rightarrow \R$, such that 
\[u(t,x)\leq C\frac{|x-z|^2}{T-t}+c(d,\epsilon,T-t).\]
\end{proposition}
\begin{proof}
First we note that for every $f\in \mathcal{A}_1(z,\epsilon)$
\[u(t,x)\leq -\log(\mathbb{P}[|X_T-z|\leq \epsilon |B_t=x]),\]
therefore
\[0\leq u(t,x)=-\log(\int_{B(z,\epsilon)}\frac{e^{-\frac{|x-y|^2}{2(T-t)}}}{(2\pi (T-t))^{\frac{d}{2}}}dy)\]
\[\leq -\log(\int_{B(z,\epsilon)}\frac{e^{-\frac{2|x-z|^2-2|y-z|^2}{2(T-t)}}}{(2\pi (T-t))^{\frac{d}{2}}})dy\leq \frac{|x-z|^2}{2(T-t)}+c(\epsilon,T-t,d).\]
\end{proof}
\begin{proposition}\label{Prop:UpperBoundOnDriftZero}
Let $\epsilon>0,$ $z=(z_1,\cdots,z_N)\in (\R^d)$ and $f\in \mathcal{A}_N(z,\epsilon)$. Define $u:[0,T]\times (\R^d)^N\rightarrow \R$  by
\begin{equation}\label{eq:EquationForUnWIthDriftZero}
    \unf(t,\vx^N):= -\frac{1}{N}\log(\mathbb{E}[e^{-Nf(T_N(X_T^1,\cdots, X_T^N))_)}|(X_t^1,\cdots, X_t^N)=\vx^N])
\end{equation}
where for each $i\in \{1,\cdots,N\}$ 
\[dX_s^i=dW_s^i  \text{ for standard i.i.d. Brownian motions.}\]
Then, we have that for some function $c(0,\infty)\times \R\times \N\rightarrow \N$,
\[0\leq \unf(t,\vx^N)\leq \frac{\mathcal{W}_2(T_{N}(\vx^N),T_N(z^N))}{2(T-t)}+c(\epsilon, T-t,d).\]
\end{proposition}
\begin{proof}
In this case the particles are completely uncoupled, therefore using the fact that 
\[\mathbb{P}\Big[\bigcap\limits_{i=1}^N \{\|X_T^i-z_i\|_1\leq \epsilon\}\Big| X_t^i=x_i\Big]\]
\[\leq \mathbb{P}\Big[W_1(T_N(X_T^1,\cdots,X_T^N),T_N(z^N))\leq \epsilon \Big| (X_t^1,\cdots,X_t^N)=x\Big],\]
and that for every $f\in \mathcal{A}_N(z^N,\delta)$
\[\unf(t,\vx^N)\leq -\frac{1}{N}\log(\mathbb{P}\Big[W_1(T_N(X_T^1,\cdots,X_T^N),T_N(z^N))\leq \epsilon \Big| (X_t^1,\cdots,X_t^N)=x\Big])\]
the result follows from Proposition \eqref{Prop:UpperBoundOnUnIndividualParticle}.
\end{proof}
\begin{proposition}\label{Prop:UpperBoundOnUnFinal}
Let $\epsilon>0$, $z=(z_1,\cdots,z_N)\in (\R^d)^N$ and $b:(\R^d)\times \mathcal{P}(\R^d)\rightarrow \R$ satisfy \eqref{assum: AssumptionOnb}. For $f\in \mathcal{A}_N(z,\epsilon)$ define $u^{N,f}:[0,T]\times (\R^d)^N\rightarrow \R$ by 

\[\unf(t,\vx^N):=-\frac{1}{N}\log(\mathbb{E}[e^{-Nf(T_N(X_T^1,\cdots, X_T^N))_)}|(X_t^1,\cdots, X_t^N)=\vx^N])\]
where for each $i\in \{1,\cdots,N\}$ we have that 
\[dX_s^i =b(X_s^i,m_s^{N})ds+dW_s^i, \text{ }\]
for standard i.i.d. Brownian motions. Then, there exists a constant $C=C(b,T,d)$ and a function $c=c(T-t,\epsilon,T,d,b)$, such that 
\[0\leq \unf(t,x)\leq C\frac{\mathcal{W}_2^2(T_N(x),T_N(z))}{T-t}+c(T-t,\epsilon,T,d,b).\]
\end{proposition}
\begin{proof}
The function $\unf$ is a solution to
\begin{equation}\label{MainEqForUn}
\begin{cases}
    -\pt \unf -\frac{1}{2}\Delta \unf +\bn \cdot D\unf +\frac{N}{2}|D\unf|^2=0 \text{ in }(0,T)\times (\R^d)^N,\\
    \unf(T,\vx^N)=f(T_N(\vx^N)).
\end{cases}
\end{equation}
Consider the function 
\[w^{N,f}(t,\vx^N):=-\log(\mathbb{E}[e^{-Nf(Y_T^1,\cdots,Y_T^N)}|(Y_t^1,\cdots,Y_t^N)_t=\vx^N])\]
where $dY_s^i=dW_s^i$ for standard i.i.d. Brownian motions. The function $w^{N,f}$ satisfies the equation 
\begin{equation*}
    \begin{cases}
-\pt w^{N,f} -\frac{1}{2}\Delta w^{N,f} +\frac{N}{2}|Dw^{N,f}|^2=0 \text{ in }(0,T)\times (\R^d)^N,\\
    w^{N,f}(T,\vx^N)=f(T_N(\vx^N)).
    \end{cases}
\end{equation*}
Thus, the function $v^N=2w^N+\|b\|_{\infty }^2(T-t)$ satisfies
\begin{equation}
    \begin{cases}
    -\pt v^N  -\frac{1}{2}\Delta v^N  +\frac{N}{4}|Dv^N|^2=\|b\|_{\infty }^2 \text{ in }[0,T)\times (\R^d)^N,\\
    v^N(T,x)=2f(T_N(x)).
    \end{cases}
\end{equation}
Moreover, from equation \eqref{MainEqForUn} we have
\[-\pt \unf -\frac{1}{2}\Delta \unf-\frac{1}{N}\sum\limits_{i=1}^N \|b\|_{\infty}^2-\frac{N}{4}|D\unf|^2+\frac{N}{2}|D\unf|^2\leq 0 \]
and so
\[-\pt \unf -\frac{1}{2}\Delta \unf +\frac{N}{4}|D\unf|^2\leq \|b\|_{\infty}^2\leq -\pt v^N-\frac{1}{2}\Delta v^N+\frac{N}{4}|Dv^N|^2,\]
which shows that $\unf\leq v^N$. Using the fact that $2f\geq f$, the result follows from Proposition \eqref{Prop:UpperBoundOnDriftZero}.

\end{proof}
\begin{proposition}\label{Prop:GradientBoundsuN}
Let $\epsilon>0$, $z=(z_1,\cdots,z_N)\in (\R^d)^N$ and $f\in \mathcal{A}_N(\epsilon,z)$. Define $\unf$ as in Proposition \ref{Prop:UpperBoundOnUnFinal}. Then, there exists a constant $C=C(T,b,d)>0$ such that 
\[|D\unf(t,\vxN)|\leq \sqrt{\frac{C}{N}(\frac{\unf(t,\vx^N)}{N(T-t)}+1)}.\]
\end{proposition}
\begin{proof}

First we find the equation satisfied by $z=\frac{1}{2}|D\unf|^2$. for points $\vx^N\in (\R^d)^N$ we will use the notation 
\[x=(x_{11},\cdots,x_{1d},\cdots,x_{N1},\cdots, x_{Nd})=(x_1,\cdots,x_N)\in (\R^d)^N.\]
For derivatives we use the notation $\unf_{i_j}:=\partial_{x_{i_j}}\un$ for $1\leq i\leq N, 1\leq j\leq d$ and for $\bn(\vx^N,T_{N}(\vx^N))=(b(x_1,T_N(\vx^N)),\cdots ,b(x_N,T_N(\vx^N)))$ we denote $b_{{i_j}}=\partial_{x_{i_ij}}b(x_i,T_N(x))$. From equation \eqref{eq: EquationUnf} we have
\[-\pt \unf_{i_j}-\frac{1}{2}\Delta \unf_{i_j}+NDu D\unf_{i_j}+\bn \cdot D\unf_{i_j}+b_{{i_j}}\cdot D_i \unf+\frac{1}{N}(D_m\bn,e_{ij})D\unf =0,\]
hence $z=\frac{1}{2}|D\unf|^2$ solves
\[-\pt z-\frac{1}{2}\Delta z+|D^2\unf|^2+ND\unf Dz+\bn Dz-2\|D_x b\|_{\infty}z-\frac{4}{N}\|D_mb\|_{\infty}^2-2z\leq 0.\]
Thus, for a constant $A>0$ which will be determined later, the function if $w=e^{At}z$, satisfies
\[-\pt w+ND\unf Dw+\bn \cdot Dw\leq -w(A-C)+\frac{C}{N}\]
where $C>0$ is a constant depending only on $b$. Thus, by choosing the constant $A=A(b,T)$ large enough, we get that for the elliptic operator 
\[L_{N}(v)=-\pt v-\frac{1}{2}\Delta v+ND\unf Dv +\bn Dv,\]
\[L_N(w)\leq \frac{C}{N}.\]
Finally, let $B>0$, $\lambda>0$ be constants to be determined later, we then have
\[L_N(\frac{\lambda }{N} \unf+\frac{B}{N}(T-t)- (T-t)w)=\frac{\lambda N}{2N}z-e^{At}z+\frac{B}{N}-(T-t)\frac{C}{N}\geq 0,\]
as long as $\lambda,B$ are large enough independently of $N$. Hence, the minimum of the function
\[\frac{\lambda }{N} \unf+\frac{B}{N}(T-t)- (T-t)w\]
is achieved at $t=T$ where
\[\unf+\frac{B}{N}(T-t)-(T-t)w\Big|_{t=T}\geq 0\]
and the result follows. 

\end{proof}
\begin{proposition}\label{Prop:SemiconcavityEstimatesuN}
Let $\unf $ be as in Proposition \ref{Prop:UpperBoundOnUnFinal}. Then, there exists a constant $C=C(b,T)>0$ such that 
\[D^2 \unf \leq \frac{C}{T-t}.\]
Furthermore, there exists a constant $C>0$ such that 
\[\pt \unf \geq -\frac{C}{T-t}.\]
\end{proposition}
\begin{proof}
For simplicity we will only show $\Delta \unf \leq \frac{C}{T-t}$ in the case $b=0$ and thus $\bn =0$. The modifications needed in the case of a drift are similar to the ones in Proposition \eqref{Prop:GradientBoundsuN}. The equation satisfied by $w=\Delta \unf $ is
\begin{equation}
    -\pt w -\frac{1}{2}\Delta w+N|D^2 \unf |^2+ND\unf  Dw =0 \text{ in }(0,T)\times (\R^d)^N.
\end{equation}
We look at the max of the function 
\[z=(T-t)w.\]
If the maximum occurred at $t<T$ we have from the equation that 
\[w+N|D^2\unf |^2\leq 0\implies w+w^2\leq 0 \implies w\leq -1\]
and the first claim follows.\\
For the second claim, from the equation satisfied by $\unf$ and Propositions \eqref{Prop:GradientBoundsuN} and \eqref{Prop:SemiconcavityEstimatesuN} we have that for some constant $C>0$ 
\[\pt \unf =-\frac{1}{2}\Delta \unf +\frac{N}{2}|D \unf|^2+\bn D\unf \geq -\frac{C}{T-t}-\frac{N}{4}|D\unf|_2^2-\frac{1}{N}\|\bn\|_2^2\geq -\frac{C}{T-t}-C.\]
\end{proof}
Finally, we carry over all the necessary estimates to the function $\un$ through the procedure outlined in the introduction of this section.
\begin{theorem}\label{thm: MainEstimates}
Let $\epsilon>0$, $z^N=(z_1,\cdots,z_N)\in (\R^d)^N$ and $b$ which satisfies \eqref{assum: AssumptionOnb}. For a given $\vx^N=(x_1,\cdots,x_N)\in (\R^d)^N$ let $X_{\cdot}^i$ be given by \eqref{eq: EquationForTheX^i's}. Then, there exists a function $c_{\epsilon,d}:\R\rightarrow \R$, such that 
\[0\leq \un(t,\vx^N)\leq c_{\epsilon,d}\Big(\frac{\mathcal{W}_2(T_N(\vx^N),T_N(z^N))+1}{T-t} \Big),\]
\[|D\un (t,\vx^N)|\leq \frac{c_{\epsilon,d}}{\sqrt{N}}\frac{\mathcal{W}_2(T_N(T_N(\vx^N),T_N(z^N)))+1}{T-t},\]
\[\pt \unf \geq -\frac{C}{T-t}.\]
\end{theorem}

\section{Convergence results} The convergence result is a standard one using the estimates above. In the rest of the section we assume that $b:\R^d\times \mathcal{P}(\R^d)\rightarrow \R$ satisfies assumption \eqref{assum: AssumptionOnb}, $X_{\cdot}^i$ satisfy \eqref{eq: EquationForTheX^i's}, $z^N\in (\R^d)^N$ and $\un :[0,T]\times (\R^d)^N\rightarrow \R$ is given by 
\[\un(t,x)=-\frac{1}{N}\log(\mathbb{P}[\mathcal{W}_2(T_N(X_T^1,\cdots,X_T^N)),T_N(z^N)\leq \epsilon |(X_t^1,\cdots, X_t^N)=z^N]).\]
First we require the following technical Lemma
\begin{lemma}\label{Lemma:LemmaForW_2Distance}
Let $z^N\in (\R^d)^N$ be such that 
\[\mathcal{W}_2(T_N(z^N),\delta_0)=M<\infty,\]
where $\delta_{0}$ is the dirac mass for the origin of $\R^d$. Then, given any $\vx^N,\vy^N\in (\R^d)^N$ and $s\in [0,1]$, we have that 
\[\mathcal{W}_{2}(T_N(s\vx^N+(1-s)\vy^N),T_N(z^N))\leq s\mathcal{W}_2(T_N(\vx^N),T_N(z^N))+(1-s)\mathcal{W}_2(T_N(\vy^N),T_N(z^N))+2\mathcal{W}_2(T_N(z^N),\delta_0).\]
\end{lemma}
\begin{proof}
We note that for any set of permutations $\sigma,\theta \in S_N$ we have
\[\|sx+(1-s)y-z\|_2\leq s\|x-z_{\sigma}\|_2+(1-s)\|y-z_{\theta}\|_2+s\|z-z_{\sigma}\|_2+(1-s)\|z-z_{\theta}\|_2\]
\[\leq s\|x-z_{\sigma}\|_2+(1-s)\|y-z_{\theta}\|_2+2\|z\|_2\]
and the claim follows.
\end{proof}
\begin{proposition}\label{Prop:LipschitzBoundsOnW2}
Let $T'<T$. Then, for every $M>0$, there exists a constant $C=C(M,b,T,T-T',d,\mathcal{W}_2(T_N(z^N),\delta_0))$ such that 
\[|\un(t,\vx^N)-\un(t,\vy^N)|\leq C\mathcal{W}_{2}(T_N(\vx^N),T_N(\vy^N)) \text{ for all }\vx^N,\vy^N \text{ such that }T_N(\vx^N),T_N(\vy^N)\in B_{\mathcal{W}_2}(T_N(z^N),M).\]
\end{proposition}
\begin{proof}
From Theorem \eqref{thm: MainEstimates}, we clearly have that such a $C>0$ as claimed in the statement exists so that 
\[\|D\un\|_2\leq \frac{C}{\sqrt{N}}.\]
We now have,
\[|\un(t,\vy^N)-\un(t,\vy^N)|\]
\[=|\int_0^1 D\un(sx_{\sigma}+(1-s)y_{\mu})\cdot (\vx^N_{\sigma }-\vy^N_{\mu})ds|\leq \int_0^1 \|D\un(t,s\vx^N_{\sigma}+(1-s)\vy^N_{\mu})\|_2\|\vx^N_{\sigma }-\vy^N_{\mu}\|_2ds, \]
we now make use of Lemma \eqref{Lemma:LemmaForW_2Distance} to obtain
\[|\un(t,\vx^N)-\un(t,\vy^N)| \leq C\frac{\|\vx^N -\vy^N\|_2}{\sqrt{N}}\]
using the invariance under permutations of $\un$ the result follows.
\end{proof}
With the above estimates the following result is standard.
\begin{proposition}\label{Prop:LiftedFunctionDefn}
Let $t\in [0,T)$. Then, the function
\[U^{N,\epsilon} (t,\mu):=\sup\limits_{x\in (\R^d)^N}\{\un(t,x)-C\frac{\mathcal{W}_2(T_N(x),T_N(z))+\mathcal{W}_2(\mu,T_N(z))+1}{T-t}\mathcal{W}_2(T_N(x),\mu)\}\]
is well defined in $\mathcal{P}_2$ and it is locally Lipschitz continuous with respect to $\W_2$. 
\end{proposition}
\subsection{Proof of Main Result}
In this subsection we will show the main results. First we fix some notation:
Given $\theta\geq 0$ and measures  $\nu,\mu_T\in \mathcal{P}_2(\R^d)$ we denote by $\mathcal{A}_{\nu,\mu_T,\theta}$ the set of controls $a$ such that
\[\pt \mu_s^a+\mathcal{L}_{\mu_s^a}^*\mu_s^a -\text{div}(a\mu_s^a)=0,\]
\[\mu^a(t)=\nu,\]
with the additional requirement that $\mathcal{W}_2(\mu_T^a,\mu_T)\leq \theta$. In the case $\theta =0$ we simply write $\mathcal{A}_{\nu,\mu_T,0}=\mathcal{A}_{\nu,\mu_T}$ while for $\theta=\infty$ we write $\mathcal{A}_{\nu}$. Before we proceed with the proof of the main theorem we need the following simple Lemma.
\begin{lemma}\label{MainLemma}
Let $\nu,\mu_T\in \mathcal{P}_2(\R^d)$. Then we have that 
\begin{equation}
    \lim\limits_{\epsilon\rightarrow 0^+}\inf\limits_{a\in \mathcal{A}_{\nu,\mu_T,\epsilon}}\Big[\frac{1}{2}\int_t^T \langle \mu_s^a,|a_s|^2\rangle ds \Big]=\inf\limits_{a\in \mathcal{A}_{\nu,\mu_T}}\Big[\frac{1}{2}\int_t^T \langle \mu_s^a,|a_s|^2\rangle ds \Big].
\end{equation}
\end{lemma}
\begin{proof}
    It is enough to consider the case where 
    \[\inf\limits_{a\in \mathcal{A}_{\nu,\mu_T,\epsilon}}\Big[\frac{1}{2}\int_t^T \langle \mu_s^a,|a_s|^2\rangle ds \Big] \leq M,\]
for some $M>0$. Furthermore, per the usual approach we consider the equivalent convex problem
\[\inf\limits_{w/m\in \mathcal{A}_{\nu,\mu_T,\epsilon}}\Big[\frac{1}{2}\int_t^T \langle \mu_s^w,\Bigg|\frac{w}{\mu_s^w}\Bigg|^2\rangle ds \Big]=\inf\limits_{a\in \mathcal{A}_{\nu,\mu_T,\epsilon}}\Big[\frac{1}{2}\int_t^T \langle \mu_s^a,|a_s|^2\rangle ds \Big]\]
where now the infimum is taken over all pairs $(\mu,w)$ such that 
\begin{equation}
\begin{cases}
     \pt \mu^{w}_s+\mathcal{L}_{\mu^{w}_s}^*\mu^{w}_s -\text{div}(w)=0,\\
     \mu^w(t) = \nu, \mathcal{W}_1(\mu_T^w, \mu_T)\leq \epsilon.
\end{cases}
\end{equation}
For each $\epsilon>0$ let $(w^{\epsilon},\mu^{\epsilon})$ be an $\epsilon-$optimal trajectory and let $\mu_T^{\epsilon}=\mu^{\epsilon}(T)$. We have that
\begin{equation}
    \begin{cases}
        \pt \mu^{\epsilon}_s+\mathcal{L}_{\mu^{\epsilon}_s}^*\mu^{\epsilon}_s -\text{div}(w^{\epsilon})=0,\\
        \mu^{\epsilon}(T)=\mu_T^{\epsilon}, \mathcal{W}_2(\mu_T,\mu_T^{\epsilon})\leq \epsilon,
    \end{cases}
\end{equation}
and thus
\[\lim\limits_{\epsilon\rightarrow 0^+}\mathcal{W}_2(\mu_T,\mu_T^{\epsilon})=0.\]
Moreover, by H\"older we have the bound 
\[\int_t^T \intd |w^{\epsilon}|dxdt=\int_t^T \intd \Bigg|\frac{w^{\epsilon}}{\sqrt{\mu^{\epsilon}}}\Bigg|\sqrt{\mu^{\epsilon}}dxdt\leq \sqrt{M}. \]
Therefore, up to subsequences $w^{\epsilon}$ has a $*-$weak limit $w$, while $\mu^{\epsilon}$ converges strongly in $\mathcal{W}_1$ to some path $\mu$. Due to the above, we may pass to the limit inside the drift $b$ to obtain 
\begin{equation}
    \begin{cases}
        \pt \mu_s+\mathcal{L}_{\mu_s}^*\mu_s -\text{div}(w)=0,\\
        \mu(T)=\mu_T.
    \end{cases}
\end{equation}
Therefore, applying this control $(\mu,w)$ shows the result.
\end{proof}
Before we proceed with the proof of the main Theorems we need to the following technical Lemma whose proof is given in the appendix.
\begin{lemma}\label{lem: Lemma for f blow up}
    Let $m\in \mathcal{P}_1(\R^d)$. Then given any $0<\epsilon<\epsilon '$ and $\delta>0$ there exists a Lipschitz and convex function $f:\mathcal{P}_1(\R^d)\rightarrow [0,\infty)$ such that 
    \begin{equation}
        \begin{cases}
            f(\nu) = 0 \text{ for all }\W_1(\nu,m)\leq \epsilon,\\
            f(\nu )\geq \frac{1}{\delta} \text{ for all }\W_1(\nu,m)\geq \epsilon'.
        \end{cases}
    \end{equation}
\end{lemma}
The above simple Lemma will allow us to establish convergence of $u^{N,f}$ as given by \eqref{eq: DefinitionOfuNF} to
\[U^f(t,\nu):=\inf\limits_{a\in \mathcal{A_{\nu}}}\Big[\frac{1}{2} \int_t^T \langle \mu_s^a,|a_s|^2\rangle ds +f(\mu_T^a) \Big].\]
Although $f$ may not be regular, as shown in \cite{daudin2023optimal}, Lipschitz and convex data are enough to show this result. Indeed, although the results in \cite{daudin2023optimal} are stated in $\T^d$, we may check that for convex data, the bound
\begin{equation}\label{eq: ConvexBoundInequality}
    u^{N,f}(t,\vx)\geq U^f(t,m_{\vx}^N),
\end{equation}
from \cite[Lemma 7.1]{daudin2023optimal} still holds in $(\R^d)^N$. For the reverse inequality, we may work as in \cite[Proposition 6.2]{daudin2023optimal} or \cite[Proposition 6.7]{cardaliaguet2022algebraic} which hold for Lipschitz data and yield convergence in $\mathcal{P}_2(\R^d)$.
\begin{theorem}\label{Thm:OptimalControlFormula1}
Let $t\in [0,T)$. The functions $U^{N,\epsilon}(t,\mu)$ converge to locally uniformly to a function $\Ue:[0,T]\times \mathcal{P}(\R^d)\rightarrow \R$, which satisfies
\[\lim\limits_{\epsilon\rightarrow 0}\Ue(t,\nu):=\inf\limits_{a\in \mathcal{A}_{\nu,\mu_T}}\Big[\frac{1}{2}\int_t^T \langle \mu_s^a,|a_s|^2\rangle ds \Big].\]

\end{theorem}
\begin{proof}
First we assume that 
\[\Ue(t,\nu)=M<\infty.\]
 For the lower bound by Lemma \ref{lem: Lemma for f blow up}, let $f:\mathcal{P}(\R^d)\rightarrow [0,\infty)$ be a convex and Lipschitz function, such that 
\begin{equation}\label{cond:Condition1Forf}
    \begin{cases}
    f(\nu) =0 \text{ if }\mathcal{W}_1(\nu,\mu_T)\leq \epsilon,\\
    f(\nu) >0 \text{ otherwise,}
    \end{cases}
\end{equation}
and the corresponding solution $u^{N,f}$ to \eqref{eq: DefinitionOfuNF}. By construction we have that 
\begin{equation}\label{eq: uNf less u}
    u^{N,f}\leq \un.
\end{equation}
Since $f$ is convex and Lipschitz by \cite{daudin2023optimal}, the functions $u^{N,f}$ converge to
\[U^f(t,\nu):=\inf\limits_{a\in \mathcal{A_{\nu}}}\Big[\frac{1}{2} \int_t^T \langle \mu_s^a,|a_s|^2\rangle ds +f(\mu_T^a) \Big].\]
Therefore letting $N\rightarrow \infty$ in \eqref{eq: uNf less u}, we obtain \begin{equation}\label{eq:LowerBoundU}
   \inf\limits_{a\in \mathcal{A_{\nu}}}\Big[\frac{1}{2} \int_t^T \langle \mu_s^a,|a_s|^2\rangle ds +f(\mu_T^a) \Big]\leq \Ue(t,\nu).
\end{equation}
By Lemma \ref{lem: Lemma for f blow up}, we now choose a sequence of convex and Lipschitz functions $f^{\delta}$, whose Lipschitz constant may blow up, satisfying \eqref{cond:Condition1Forf} such that 
\[f^{\delta }(\nu)\geq \frac{1}{\delta} \text{ for all }\nu \text{ such that }\mathcal{W}_2(\nu,\mu_T)\geq 2\epsilon.\]
Thus, from Lemma \eqref{MainLemma} and inequality \eqref{eq:LowerBoundU} by letting $\delta\rightarrow 0$ and then $\epsilon\rightarrow 0$ we obtain the lower bound. For the upper bound the claim follows in a similar way by considering a sequence of functions $f^{\delta}\geq 0$ such that 
\begin{equation}
    \begin{cases}
    f^{\delta }(\nu) = 0 \text{ if }\mathcal{W}_2(\nu,\mu_T)\leq \frac{\epsilon}{4},\\
    f^{\delta}(\nu)\geq \frac{1}{\delta} \text{ if }\mathcal{W}_2(\nu,\mu_T)\geq \frac{\epsilon}{2},
    \end{cases}
\end{equation}
and a sequence $y^N\in (\R^d)^N$ such that 
\[\un(t,y^N)\rightarrow \Ue(t,\nu).\]
By comparison of the solutions $\un$ and the corresponding $u^{N,f^{\delta}}$ of \eqref{eq: DefinitionOfuNF} we obtain that for $\delta$ small enough 
\[\Ue(t,\nu)\leq \inf\limits_{a\in \mathcal{A}_{\nu,\mu_T,2\frac{\epsilon}{4}}}\Big[\frac{1}{2} \int_t^T \langle \mu_s^a,|a_s|^2\rangle ds  \Big]\]
which again by Lemma \eqref{MainLemma} proves the result.\\
Finally in the case where $\lim\limits_{\epsilon\rightarrow 0}\Ue(t,\nu)=\infty$, we may use the argument for the upper bound above.
\end{proof}
\begin{theorem}
Assume that $b$ satisfies \eqref{assum: AssumptionOnb}. Let $\nu\in \mathcal{P}_2(\R^d)$ and $x^N=(x_1,\cdots,x_N)\in (\R^d)^N$ be such that 
\[\lim\limits_{N\rightarrow \infty}\mathcal{W}_2(T_N(x^N),\nu)=0.\]
Furthermore, let $\vX_{\cdot}^N$ be given by
\begin{equation*}
    \begin{cases}
    dX_s^{i,N}=b(X_s^{i,N},m_{\vX_s^N}^N)dt+dW_t^i,\\
    X_t^{i,N}=x_i.
    \end{cases}
\end{equation*}
 Then, the process at the terminal time $m_{\vX_T^N}^N$ satisfies a LDP on $\mathcal{P}_2(\R^d)$ with good rate function 
\[I(t,\mu_T)=\inf\limits_{a\in \mathcal{A}_{\nu,\mu_T}}\Big[ \frac{1}{2}\int_t^T \langle \mu_s^a,|a_s|^2\rangle ds \Big].\]
Finally, the process $\mu_{\cdot}^N$ satisfies a weak-LDP in $C([0,T];\mathcal{P}_2(\R^d))$ with rate function 
\[I(\mu(\cdot))=\inf\limits_{a\in \mathcal{B}}\Big[\frac{1}{2} \int_t^T \langle \mu_s^a,|a_s|^2\rangle ds \Big],\]
where $a\in \mathcal{B}$ are all the controls for which
\begin{equation*}
    \begin{cases}
    \pt \mu(t)+\mathcal{L}_{\mu(t)}^*\mu(t)-\text{div}(a\mu)=0,\\
    \mu(0)=\nu.
    \end{cases}
\end{equation*}
\end{theorem}
\begin{proof}
The first claim follows from Theorem \eqref{Thm:OptimalControlFormula1}. The second claim is a corollary, see for example \cite{feng2006large}.
\end{proof}
\section{Common Noise}
In this section we establish a weak LDP in the presence of common noise. First we recall the setup. Our dynamics are given by 
\begin{equation}
\begin{cases}
        d X_s^{i,N} = b(X_s^{i,N}, m_{\vX_s^N}^N)ds+dW_s^i+\sqrt{\alpha_0} dW_s^0,\\
        X_t^{i,N}= x^i,
\end{cases}
\end{equation}
where in the above $W_{\cdot}^i,W_{\cdot}^0$ are standard i.i.d. Brownian motions and $x^i \in (\R^d)$. Similar to the case of $\alpha_0 =0$, fix $\mu_T\in \mathcal{P}(\R^d)$, $z^N\in (\R^d)^N$ satisfying \ref{eq: DefnOfz^N} and $\epsilon>0$. In order to identify the rate function for the empirical distribution of
\[m_{\vX_{t}^N}^N:=\frac{1}{N}\sum\limits_{i=1}^N \delta_{X_t^{i,N}},\]
we look at the functions
\begin{equation}\label{defn: DefinitionOfCommonNoiseMainFunction}
    \unt (t,x_1,\cdots,x_N)= -\frac{1}{N}\log\Big(\mathbb{P}\Big[\mathcal{W}_1(m_{\vX_T^N}^N,T_N(z^N))\leq \epsilon | (X_t^1,\cdots,X_t^N)=(x_1,\cdots,x_N) \Big]\Big).
\end{equation}
We emphasize that in the above, the probability measure is with respect to all the underlying Brownian motions $W^0,W^1,\cdots, W^N$. Formally the functions $\unt$ satisfy 
\begin{equation}\label{eq: CommonNoisePDENonScale}
    \begin{cases}
        -\pt \unt -\frac{1}{2}\Delta \unt -\frac{\alpha_0}{2} \sum\limits_{i=1}^N \sum\limits_{j=1}^N \text{tr}\Big(D_{ij}\unt \Big)+\bn \cdot D\unt+\frac{N}{2}\Big|D\unt \Big|^2+\frac{N\alpha_0}{2}\Big|\sum\limits_{i=1}^N D_i \un\Big|^2 =0 \text{ in }(0,T)\times (\R^d)^N,\\
        \unt(T,x) = \begin{cases}
            0 \text{ if }\mathcal{W}_1(T_N(x),T_N(z^N))\leq \epsilon,\\
            +\infty \text{ otherwise}.
        \end{cases}
    \end{cases}
\end{equation}
Equation \eqref{eq: CommonNoisePDENonScale} just like \eqref{eq: DefinitionOfu^N,epsilon} is elliptic, however unlike \eqref{eq: DefinitionOfu^N,epsilon}, not all terms have the same scaling in terms of $N$. In particular the term 
\[\frac{\alpha_0N}{2}\Big|\sum\limits_{i=1}^ND_i\unt \Big|^2\]
is of order $N$. To illustrate this, assume that $\unt(t,\vx) = U^N(t,T_N(\vx))$ for a smooth function $U^N:[0,T)\times \mathcal{P}(\R^d)\rightarrow \R$. Then, it follows that
\[\frac{\alpha_0N}{2}\Big|\sum\limits_{i=1}^ND_i\unt \Big|^2 = \frac{\alpha_0N}{2}\Big|\langle D_mU^N(t,T_N(\vx), \cdot), T_N(\vx)\rangle \Big|^2.\]
Therefore unlike the case of $\alpha_0= 0$, a different approach is required, which we describe next. For more heuristics on this issue we refer to Remark \ref{rmk: Heuristics common noise PDE}.\\

Our approach is motivated by the one carried out by F. Delarue, D. Lacker, and K. Ramanan in \cite{Lacker1}. In fact, to an extent we are rephrasing their argument in a PDE setting. The main idea as explained in \cite{Lacker1}, is to "freeze" the common noise $W^0$ and treat it as a deterministic path. On a heuristic level, the approach will be as follows:\\
\begin{itemize}
    \item First we derive the rate function for the empirical distribution of 
    \begin{equation}\label{eq: Definition of Y^phi}
    Y_t^{i,N} = x^i+ \int_{t_0}^t b(Y_s^{i,N},m_{\vY^N}^N)ds +(W_t^i-W_{t_0}^i) +\sqrt{\alpha_0}(\phi(t)-\phi(t_0)),
\end{equation}
where $\phi \in C^0([0,T];\R^d)$. Let $U^{\phi}$ be this rate function. We may think of this, as the rate function of the original empirical distribution, conditioned on $\{W^0 =\phi\}$.\\
\item Once the above has been established, we will show that the rate function of the original empirical distribution, call it $U$, is given by 
\[U(t,\mu) = \inf\limits_{\phi \in C^0}U^{\phi}(t,\mu). \]
\end{itemize}
The first step is addressed in subsection \ref{subsection: Prerequisities for Common noise}, while the later in subsection \ref{subsection: Rate function Common Noise}. 
\begin{rmk}
    The idea of fixing the common noise, has been the main tool in the PDE approach to MFG with common noise. However, it is important to point out that unlike the case of $\alpha_0 =0 $, equation \ref{eq: CommonNoisePDENonScale} is not of the typical form studied in the $N-$player case for MFG or MFG of control, see for example P. Cardaliaguet, S. Daudin, J. Jackson, and P. Souganidis in \cite{cardaliaguet2022algebraic} and P. Cardaliaguet, F. Delarue, J.-M. Lasry, and P.-L. Lions in \cite{cardaliaguet2019master}.
\end{rmk}
\subsection{Prerequisites}\label{subsection: Prerequisities for Common noise}
In this subsection we derive the rate function for the empirical measure of $\vY^{\phi,N}$ given by \ref{eq: Definition of Y^phi}. First we introduce some notation: For $\phi \in C^0([0,T]; \R^d)$ and $\vx = (x_1,\cdots, x_N)\in (\R^d)^N$, let $\vY^{\phi,N}=(Y^{1,\phi}, \cdots, Y^{N,\phi})$ be given by 
\begin{equation}\label{process: Y^phi}
    Y_t^{\phi} = x_i+ \int_{t_0}^t b(Y_s^{i,n\phi},T_N(\vY_s^{\phi,N}))ds +(W_t^i-W_{t_0}^i) +\sqrt{\alpha_0}(\phi(t)-\phi(t_0)).
\end{equation}
In order to have these dynamics in the same form as the case $\alpha_0 = 0$, we consider instead  the processes 
\begin{equation}
    Z_t^{i,\phi} = Y_t^{i,\phi} - \sqrt{\alpha_0}(\phi(t)-\phi(t_0)),
\end{equation}
that satisfy
\begin{equation}
    \begin{cases}
        dZ_t^{i,\phi} = \tbphi (t,Z^{i,\phi}, m_{\vZ_t}^{N,\phi})dt +dW_t^i,\\
        Z_{t_0}^{i,\phi} = x_i.
    \end{cases}
\end{equation}
Where in the above we adopted the notation 
\begin{equation}
    \tbphi (t,x,\vy) = b(x+\sqrt{a_0}(\phi(t)-\phi(t_0)), T_N(y_1+\sqrt{\alpha_0}(\phi(t)-\phi(t_0)),\cdots, y_N+\sqrt{\alpha_0}(\phi(t)-\phi(t_0)))),
\end{equation}
and for later use, we fix the notation
\begin{equation}\label{defn: New drift with time}
    \Tilde{b}^{\phi,N}(t,\vx)=(\tbphi (t,x_1,\vx),\cdots, \tbphi(t,x_N,\vx))
\end{equation}
where $\vx = (x_1,\cdots, x_N)\in (\R^d)^N$. Similar to the approach in the case $\alpha_0=0$, given $\epsilon>0$, $\mu_T\in \mathcal{P}(\R^d)$ and $z^N$ satisfying \ref{eq: DefnOfz^N},  we consider the functions
\[w^{N,\phi,\epsilon}(t,\vx) = -\frac{1}{N}\log(\Prob\Big[ \W_1(T_N(\vZ_T^{N,\phi}),T_N(z^N))\leq \epsilon \Big| \vY_t= \vx \Big]),\]
which satisfy 
\begin{equation}\label{eq: EquationForCommonNoiseNoPhi}
    \begin{cases}
        -\pt w^{N,\phi,\epsilon} -\frac{1}{2}\Delta w^{N,\phi,\epsilon} +\frac{N}{2}|D w^{N,\phi,\epsilon}|^2+\Tilde{b}^{\phi, N}\cdot D w^{N,\phi,\epsilon} = 0 \text{ in }(0,T)\times (\R^d)^N,\\
        w^{N,\phi,\epsilon} (T,\vx)= \begin{cases}
            0 \text{ if }\W_1(T_N(\vx),T_N(z_1^N -\sqrt{a_0}\phi(T),\cdots ,z_N^N-\sqrt{a_0}\phi(T)))\leq \epsilon ,\\
            +\infty \text{ otherwise.}
        \end{cases}
    \end{cases}
\end{equation}

Now we define our candidate for the rate function. Recall that for $x\in \R^d$, we denoted by $\tau_x :\R^d\rightarrow \R^d$
the translation map $\tau_x(y) = y -x$.
Given $\mu\in \mathcal{P}_2(\R^d)$ and $\phi \in C^0([0,T];\R^d)$, let 
\[ \Tilde{\mu}^{\phi}(t) = \mu \circ \tau_{\sqrt{\alpha_0}\phi (t) } .\]
Finally, if $I, I_{\text{path}}$ are the rate functions given by \eqref{eq: RatefunctionInTheoremTerminalTime} and \eqref{eq: FormulaForRateFunction} respectively, we let $\Tilde{I}$ and $\Tilde{I}_{\text{path}}$ be the rate functions where the drift in the operator $\mathcal{L}$ is replaced by $\Tilde{b}$. Then our candidate rate functions are 
\[U^{\phi}(t_0,\nu) = \Tilde{I}(t_0, \Tilde{\nu}^{\phi}))\]
and 
\[U_{\text{path}}(\mu )= \Tilde{I}_{\text{path}}(\Tilde{\mu}^{\phi}).\]

We may now state the main result of this subsection.
\begin{theorem}\label{thm: LDP for Y^{phi}}
    Let $b$ satisfy \eqref{assum: AssumptionOnb}. Moreover, let $\vx^N =(x_1,\cdots, x_N)\in (\R^d)^N$ and  $\vY^{\phi}$ be given by \eqref{process: Y^phi}. Then, the process
    \[m_{\vY_{\cdot}}^{\phi}\in C([0,T]; \mathcal{P}_1(\R^d)),\]
    satisfies an LDP with rate function $U_{\text{path}}^{\phi}$. Finally, the terminal time process $m_{\vY_T^{\phi}}^N$ satisfies an LDP with rate function $U^{\phi}$.
\end{theorem}
\begin{proof}
    The result follows from Theorem \ref{thm:MainTheorem} applied to the process $\vZ^{\phi}$. We note that the estimates obtained in the case of $\alpha_0 = 0$, are still valid when introducing the time dependence in \eqref{process: Y^phi}. To obtain the rate function for $\vY^{\phi}$, we note that
    \[m_{\vY_{t}^{N,\phi}} = m_{\vZ_t^{N,\phi}}\circ \tau_{\sqrt{\alpha_0}\phi },\]
    and the result follows.
\end{proof}
\begin{rmk}
    In the case $\phi \in H^1([0,T])$ for $U_{\text{path}}$, we may check that 
    \begin{equation}
    U_{\text{path}}^{\phi}(\mu(\cdot )):= \inf\limits_{\alpha }\Big[\frac{1}{2}\int_0^T \langle \mu, |\alpha_s|^2 \rangle ds \Big],
\end{equation}
where the infimum is taken over all controls $\alpha$ (if any) such that 
\begin{equation}
    \begin{cases}
        \pt \mu-\frac{1}{2}\Delta \mu-\dive(\mu b) -\dive(\mu \dot{\phi}) -\dive(\alpha \mu ) =0,\\
        \mu(0)=\nu.
    \end{cases}
\end{equation}
While when $\phi \in C^0([0,T];\R^d)$ we have that 
\begin{equation}
    U^{\phi}(\mu(\cdot )):= \inf\limits_{\alpha \in \mathcal{A}_{\nu,\mu}^{\phi}}\Big[\frac{1}{2}\int_t^T \langle \nu_s^{\phi}, |\alpha_s|^2 \rangle ds \Big],
\end{equation}
where the infimum is taken over all controls $\alpha$ such that 
\begin{equation}
    \begin{cases}
        \pt \nu-\frac{1}{2}\Delta \nu -\dive(\nu b)  -\dive(\alpha \nu  ) =0,\\
        \nu (0)=\nu\circ \tau_{\sqrt{\alpha_0}\phi},
    \end{cases}
\end{equation}
with $\nu^{\phi} = \mu\circ \tau_{\sqrt{\alpha_0}\phi}$.
\end{rmk}

\subsection{Rate function for common noise}\label{subsection: Rate function Common Noise}
In this subsection we will prove the main Theorem. The following simple Lemma, whose proof is given in the Appendix, justifies the approach of "freezing" the common noise.
\begin{lemma}\label{lem: ReplaceW^0Byphi}
    Assume that $b$ satisfies \ref{assum: AssumptionOnb}. Let $\vX_t$ satisfy
    \begin{equation}
            X_t^i = x_0^i+ \int_{t_0}^t b(X_s^i,T_N(\vX_s))ds +(W_t^i-W_{t_0}^i) +\sqrt{\alpha_0}(W_t^0-W_{t_0})
    \end{equation}
    and let $\vY_t$ be given by 
    \begin{equation}
            Y_t^i = x_0^i+ \int_{t_0}^t b(Y_s^i,T_N(\vY_s))ds +(W_t^i-W_{t_0}^i) +\sqrt{\alpha_0}(\phi(t)-\phi(t_0)).
    \end{equation}
    Then, there exists a constant $C>0$, independent of $N$, such that for all $\delta>0$, on the event $\{\|W^0-\phi\|_{\infty}<\delta\}$ we have that 
    \begin{equation}
        \label{eq: DeltaDistanceWasserstein1}
        \sup\limits_{s\in [t_0,T]}\W_1(T_N(\vY_s),T_N(\vX_s))\leq C\delta.
    \end{equation}
\end{lemma}

Before we proceed with the proof of the main Theorem we describe the general strategy. The technique described above of "freezing" the common noise $W^0$ near a $\phi \in C^0([0,T];\R^d)$ is essentially describing the rate function for the pair 
\[\Big( m_{\vX_{\cdot}^N}^N, W_{\cdot}^0\Big)\in C([0,T]; \mathcal{P}_1(\R^d))\times C^0([0,T]; \R^d) .\]
This is also the approach in \cite{Lacker1}. However in that work, a Sanov Type theorem is used. Once a LDP has been obtained for the pair, the authors in \cite{Lacker1} apply a contraction principle to obtain a rate function for $m_{\vX_{\cdot}^N}^N$. The difficulty in our approach, stems from the fact that we only obtain a weak LDP for the pair $\Big( m_{\vX_{\cdot}^N}^N, W_{\cdot}^0\Big)$. Moreover, as explained in \cite{Lacker1}, due to the lack of compactness of the rate function for $ m_{\vX_{\cdot}^N}^N$, we cannot hope for a good rate function for the pair $\Big( m_{\vX_{\cdot}^N}^N, W_{\cdot}^0\Big)$ in order to obtain a full LDP. Therefore, in order to infer a weak LDP for $m_{\vX_{\cdot}^N}^N$ from that of $\Big( m_{\vX_{\cdot}^N}^N, W_{\cdot}^0\Big)$, we would need the knowledge that when $m_{\vX_{\cdot}^N}^N$ lies in a compact set $K\subset C^0([0,T];\mathcal{P}_1(\R^d))$, then $W^0$ also lies in a compact set of $C([0,T];\R^d)$. Such a claim however, appears to be false. To overcome this issue, we rely on the fact that 
\[\frac{1}{N}\sum\limits_{i=1}^N W_{t}^i\]
satisfies an LDP in $C^0([0,T];\R^d),$ with a good rate function by Freidlin-Wentzel-Schilder’s Theorem. Thus on the event that $\frac{1}{N}\sum\limits_{i=1}^N W_{t}^i$ lies in a compact set and $\{m_{\vX_{\cdot}^N}^N\in K\}$ we do in fact obtain that $W^0$ also lies in a compact set and the LDP for the pair may be used. Finally, we note that the LDP for $\frac{1}{N}\sum\limits_{i=1}^NW_t^i$ may be shown through PDE methods.
\begin{theorem}
    Let b satisfy \ref{assum: AssumptionOnb}. Let $\vx^N = (x_1^N,\cdots, x_N^N)\in (\R^d)^N$ and $\vX_{\cdot}^N = (X_{\cdot}^{1,N},\cdots ,X_{\cdot}^{N,N})$ be given by 
    \begin{equation}
\begin{cases}
        d X_t^{i,N} = b(X_t^{i,N},m_{\vX_t^N}^N)dt+dW_t^i+\sqrt{\alpha_0} dW_t^0,\\
        X_0^{i,N}= x^{i,N}.
\end{cases}
\end{equation}
Finally, assume that for some measure $\nu \in \mathcal{P}_2(\R^d)$, we have that
\[\lim\limits_{N\rightarrow \infty} \mathcal{W}_2(T_N(\vx^N), \nu) = 0.\]
Then, the path of the empirical measure 
\[m_{\vX_{\cdot}^N}^N \in C([0,T];\mathcal{P}_1(\R^d)),\]
satisfies a weak LDP in $C([0,T];\mathcal{P}_1(\R^d))$, with rate function
\begin{equation}
    \label{defn: U_{path}}
    U_{\text{path}}(\mu_{\cdot }) =\inf\limits_{\phi \in C^0}U_{\text{path}}^{\phi}(\mu_{\cdot}).
\end{equation}
Moreover, the empirical distribution at the terminal time $\mu_{\vX_T^N}^N$, sastisfies a weak LDP with rate function 
\[U(t_0,\mu_T)= \inf\limits_{\phi \in C^0}U^{\phi}(t_0,\mu_T).\]
\end{theorem}
\begin{proof}
We only provide the proof for the empirical density as a path, since the proof for the marginals is almost identical. First we establish the lower bound for open sets. Fix $\delta>0$ and a path $\phi \in C^0([0,T]; \R^d)$. By the results in \cite{stroock1972support} (Lemma 3.1) we have that 
\begin{equation}\label{eq: POsitive Brownian }
    \mathbb{P}\Big[\sup\limits_{[t_0,T]}\|W_{\cdot}^0-\phi(\cdot)\|_{\infty}<\delta \Big]=C(\delta, d, \phi)>0.
\end{equation}
Consider the process $\vY^{N,\phi}$ given by 
\begin{equation}
     Y_t^{i,\phi} = x_i^{N}+ \int_{t_0}^t b(Y_s^{i,\phi},T_N(\vY_s^{N,\phi}))ds +(W_t^i-W_{t_0}^i) +\sqrt{\alpha_0}(\phi(t)-\phi(t_0)).
\end{equation}
Given $\epsilon>0$ and $\mu(\cdot) \in C([0,T]; \mathcal{P}(\R^d))$, let
\[O^{\epsilon}:= \{m(\cdot)\in C([0,T]; \mathcal{P}(\R^d)): \sup\limits_{0\leq t\leq T}\mathcal{W}_1(m(t), \mu(t))<\epsilon\}.\]
From Lemma \ref{lem: ReplaceW^0Byphi} there exists a constant $C>0$, independent of $N$, such that conditioned on the set $\{\|W^0 -\phi\|_{\infty} <\epsilon\}$, we have
\[\{m_{\vY^{N,\phi}} \in O^{\epsilon}\}\subset \{m_{\vX^{N}} \in O^{2C\epsilon}\}.\]
It then follows that
\begin{align*}
    \mathbb{P}\Big[ m_{\vX^N}^N \in O^{C\epsilon} | \vX_{0}^N = \vx^N \Big] &\geq
    \mathbb{P}\Big[m_{\vX^N}^N \in O^{C\epsilon} | \vX_{0}^N = \vx^N, \|W^0-\phi\|_{\infty}<\epsilon \Big]\mathbb{P}\Big[ \|W^0-\phi\|_{\infty}<\epsilon\Big]\\
    &\geq \mathbb{P}\Big[m_{\vY^{N,\phi}}^N \in O^{\epsilon} | \vY_{0}^{N,\phi} = \vx^N \Big]\mathbb{P}\Big[ \|W^0-\phi\|_{\infty}<\epsilon\Big].
\end{align*}
Hence, from Theorem \ref{thm: LDP for Y^{phi}} it follows that 
\begin{equation}
    \liminf\limits_{\epsilon \rightarrow 0^+}\liminf\limits_{N\rightarrow \infty} \frac{1}{N}\log\Big(\mathbb{P}\Big[ m_{\vX^N}^N \in O^{C\epsilon} | \vX_{0}^N = \vx^N \Big] \Big) \geq -U^{\phi}(\mu(\cdot)),
\end{equation}
where in the above we used \eqref{eq: POsitive Brownian }. Thus we have 
\[\liminf\limits_{\epsilon \rightarrow 0^+}\liminf\limits_{N\rightarrow \infty} \frac{1}{N}\log\Big(\mathbb{P}\Big[ m_{\vX^N}^N \in O^{\epsilon} | \vX_{0}^N = \vx^N \Big] \Big) \geq - \inf\limits_{\phi \in C^0}U^{\phi}(\mu(\cdot)),\]
which establishes the lower bound for open sets. Note that the same argument given for the open balls, may be used to show that for any $\phi \in C^0$ and $\mu \in C([0,T];\mathcal{P}(\R^d))$
\[\lim\limits_{\epsilon \rightarrow 0^+}\lim\limits_{N\rightarrow\infty}\frac{1}{N}\log(\mathbb{P}\Big[ \W_1(m_{\vX_{\cdot}^N}^N,\mu(\cdot))\leq \epsilon , \|W^0-\phi\|\leq \epsilon \Big])= -U^{\phi}(\mu).\]
The above shows by Proposition \ref{prop: EpsilonZeroJustification}, that the pair $(m_{\vX_{\cdot}^N}^N,W^0)$ satisfies a weak LDP in $C([0,T]; \mathcal{P}(\R^d))\times C^0([0,T]; \R^d)$ with rate function $U^{\phi}$.\par

Next we obtain the upper bound. For ${\lambda} >0$ we define the set
\[H_{\lambda}:=\{\phi \in C^0([0,T]): \sup\limits_{s\neq t }\frac{|\phi(t)-\phi(s)|}{|t-s|^{\frac{1}{4}}}\leq \lambda \}.\]
Note that $H_{\lambda}$ is a compact subset of $C^0([0,T];\R^d)$ and moreover for $\lambda>0$, any standard Brownian motion $B_{\cdot}$ satisfies
\begin{equation}
    \mathbb{P}[B_{\cdot }\in H_{\lambda}] >0.
\end{equation}
Fix $K\subset C^0([0,T]; \R^d)$ a compact subset. To keep the presentation simple, in the following we drop the conditioning on the initial data. We have that
\begin{align}\label{eq: BoundFor m_X^N in K}
    \mathbb{P}\Big[m_{\vX^N}^N \in K\Big]\leq \mathbb{P}\Bigg[m_{\vX^N}^N \in K, \frac{1}{N}\sum\limits_{i=1}^N W^i \in H_{\lambda}\Bigg]+\mathbb{P}\Bigg[\frac{1}{N}\sum\limits_{i=1}^N W^i \in \Big( H_{\lambda}\Big)^c\Bigg] .
\end{align}

By Freidlin-Wentzel-Schilder's Theorem we have that 
\[\limsup\limits_{N\rightarrow \infty}\frac{1}{N}\log\Bigg( \mathbb{P}\Bigg[ \frac{1}{N}\sum\limits_{i=1}^N W^i \in \Big(H_{\lambda}\Big)^c \Bigg]\Bigg) \leq -\inf\limits_{\phi\in \overline{H_{\lambda}^c}} \Big(\frac{1}{2}\|\phi \|_{H^1}^2 \Big).\]
Moreover, note that if for some $s_0< t_0$ we had
\[\frac{|\phi(t_0)-\phi(s_0)|}{|t_0-s_0|^{\frac{1}{4}}}> {\lambda}\]
then 
\[ \|\phi\|_{H^1}\geq \frac{{\lambda}}{\sqrt{T}}.\]
Hence,
\begin{equation}\label{eq: Infimum Over H_lambda}
    \limsup\limits_{N\rightarrow \infty}\frac{1}{N}\log\Big( \mathbb{P}\Big[ \frac{1}{N}\sum\limits_{i=1}^N W^i \in H_{\lambda} \Big]\Big) \leq -\frac{{\lambda}^2}{2T}.
\end{equation}
Continuing with the analysis, we claim that there exists a compact set $F\subset C^0([0,T]; \R^d)$ that may depend on ${\lambda}$ and $K$, but is independent of $N$, such that 
\begin{equation}\label{eq: GlobalCompactSetForW^0}
    \{m_{\vX^N}^N \in K\}\cap \{\frac{1}{N}\sum\limits_{i=1}^N W_t^i \in H_{\lambda}\}\subset \{W^0 \in F\} \text{ for all }N\in \N.
\end{equation}
To this direction, since $K$ is compact, we note that if $\nu(\cdot) \in K$, then by continuity of 
\[(\nu_{t})_{0\leq t\leq T} \rightarrow \Big(\int_{\R^d} x d\nu_t(x) \Big)_{0\leq t\leq T}, \]
the set 
\[S:= \{\Big(\int_{\R^d} x d\nu_t(x) \Big)_{0\leq t\leq T} : \nu(\cdot ) \in K\}\subset C([0,T];\R^d),\]
is compact and thus bounded and equicontinuous. Moreover its upper bound and modulus of continuity are independent of $N$. This implies that 

\[\frac{1}{N}\sum\limits_{i=1}^N x_i^N+ \int_0^t \langle b(\cdot, m_{\vX_s^N}^N), m_{\vX_s^N}^N\rangle  ds +\frac{1}{N}\sum\limits_{i=1}^N W_t^i +W_t^0 =\langle x, m_{\vX_t^N}^N \rangle \in S.\]
Since the modulus for $S$ and $H_c$ are independent of $N$ and by our assumptions on $b$ so is $\int_0^t \langle b(\cdot, m_{\vX_s^N}^N), m_{\vX_s^N}^N\rangle  ds$ when $m_{\vX_{\cdot}^N}^N \in K$,  the claim follows. Going back to \eqref{eq: BoundFor m_X^N in K}, we now have that 
\[\limsup\limits_{N\rightarrow \infty}\frac{1}{N}\log\Big(\mathbb{P}\Big[m_{\vX^N}^N \in K \Big]\Big)\leq \max\{ \limsup\limits_{N\rightarrow \infty} \frac{1}{N}\log\Big(\mathbb{P}\Big[ m_{\vX_{\cdot}^N}^N \in K, W^0 \in F \Big] \Big), -\frac{\lambda^2}{2T}\} \]
\[\leq \max\{-\inf\limits_{\nu \in K}\inf\limits_{\phi \in C^0([0,T];\R^d)} U^{\phi}(\nu), -\frac{\lambda^2}{2T}\}.\]
Where in the first inequality we used \eqref{eq: Infimum Over H_lambda} and \eqref{eq: GlobalCompactSetForW^0}, while in the last inequality we used the weak LDP for $(m_{\vX_{\cdot}^N}^N, W^0)$ and the fact that $\inf\limits_{\phi \in F} U^{\phi} \geq \inf\limits_{\phi \in C^0([0,T]; \R^d)} U^{\phi}$. Letting $\lambda\rightarrow \infty$ yields
\[\limsup\limits_{N\rightarrow \infty}\frac{1}{N}\log\Big(\mathbb{P}\Big[m_{\vX^N}^N \in K \Big]\Big)\leq -\inf\limits_{\nu \in K}\inf\limits_{\phi \in C^0([0,T];\R^d)} U^{\phi}(\nu),\]
which proves the result.
\end{proof}

\begin{rmk}\label{rmk: Heuristics common noise PDE}
    Here we briefly discuss some heuristics regarding the limit of \eqref{eq: CommonNoisePDENonScale} and the rate function obtained. Let $G^{\epsilon}:(\R^d)^N\rightarrow [0,\infty]$ such that 
    \begin{equation}
        G^{\epsilon}(\vx)=
        \begin{cases}
        0 \text{ if }\W_1(T_N(\vx), T_N(z^N))\leq \epsilon,\\
        +\infty \text{ otherwise.}
        \end{cases}
    \end{equation} 
    The control formulation of $\unt$ give us the formula 
    \[\unt (t,\vx)= \inf\limits_{(a^1,\cdots, a^N,\beta^0)} \mathbb{E}\Big[ \int_t^T \Big(\frac{1}{2N}\sum\limits_{i=1}^N |a_s^i|^2\Big) +\frac{1}{2\alpha_0N} |\beta_s^0|^2  ds +G^{\epsilon}(\Bar{Z}_T^N)   \Big]\]
    where the infimum is taken over all controls such that 
    \[dZ_s^{i,N} = (a_s^i+\beta_s^0+b(Z^{i,N}, m_{\Bar{Z}_s^N}^N))ds +dW_s^i+\sqrt{\alpha_0}dW^0 \]
    \[Z_t = x^i.\]
    Formally passing to the limit in this expression (assuming gradient bounds), the cost of the common control $\beta^0$ disappears. This is also evident in our description of $U_{\text{path}}$ as there is no cost for choosing $\phi \in C^0([0,T];\R^d)$. However, what is not immediately clear in the limit of $\unt$, is the fact that in the limit, the common control will only depend on time.

\end{rmk}
\section{Appendix}
Proof of Lemma \ref{lem: ReplaceTheLimit}. 
\begin{proof}
    Let $\Delta_t =\frac{1}{N}\sum\limits_{i=1}^N |Y_t^i-X_t^i|$. We have that, on the event $\{\|W^0-\phi\|_{\infty}<\delta\}$,
    \[|\Delta_t|\leq \int_t^T\frac{1}{N} \sum\limits_{i=1}^N|b(X_s^i,T_N(\vX_s))-b(Y_s^i,T_N(\vY_s))|ds+\sqrt{\alpha_0}|(\phi(t)-W_t^0)-(\phi(t_0)-W_{t_0}^0)|\]
    \[\leq C\int_t^T \frac{1}{N}\sum\limits_{i=1}^N |Y_s^i-X_s^i| +\mathcal{W}_1(T_N(\vX_s,\vY_s))ds+\sqrt{a_0}\delta \]
    \[= C\int_t^T |\Delta_s|+\mathcal{W}_1(T_N(\vX_s),T_N(\vY_s))ds+\sqrt{a_0}\delta,\]
    where in the second inequality we used our assumptions on $b$. Thus by Gr\"onwall we have that for some constant $C>0$ independent of $N$
    \[\sup\limits_{s\in [t_0,T]}\mathcal{W}_1(T_N(\vX_s),T_N(\vY_s))\leq C\delta.\]
\end{proof}
Proof of Lemma \ref{lem: Lemma for f blow up}.
\begin{proof}
    We define the function $g:\mathcal{P}_1(\R^d)\rightarrow \R$ by 
    \[g(\nu) = \W_1(\nu, m_0).\]
    This function is clearly convex and globally Lipschitz in $\W_1$. Let $h: [0,\infty)\rightarrow [0,\infty)$ be an increasing, locally Lipschitz and convex function such that 
    \begin{equation}
        \begin{cases}
            h(x) = 0 \text{ if }x\leq \epsilon,\\
            h(x) \geq \frac{1}{\delta} \text{ if }x\geq \epsilon '.
        \end{cases}
    \end{equation}
    Then the function $f(\nu) = h(g(\nu))$ satisfies the assumptions.
\end{proof}

\bibliographystyle{siamplain}
\bibliography{references}

\end{document}